\documentclass[draft]{amsart}

\usepackage{amsthm} 
\usepackage{amssymb} 
\usepackage{enumerate} 

\usepackage{hyperref}

\numberwithin{equation}{section}

\theoremstyle{plain}
\newtheorem{lemma}{Lemma}[section]
\newtheorem{theorem}[lemma]{Theorem}
\newtheorem{proposition}[lemma]{Proposition}

\newtheorem{question}[lemma]{Question}

\theoremstyle{definition}

\theoremstyle{remark}
\newtheorem{remark}[lemma]{Remark} 

\renewcommand{\dim}{\operatorname{dim}}
\newcommand{\depth}{\operatorname{depth}}

\newcommand{\hh}{\operatorname{H}}

\newcommand{\flatdim}{\operatorname{flat\,dim}}
\renewcommand{\le}{\leqslant}
\renewcommand{\ge}{\geqslant}

\newcommand{\kos}[2]{\operatorname{K}[#1;#2]}
\newcommand{\length}{\ell}
\newcommand{\lol}{\ell\ell}
\newcommand{\Spec}{\operatorname{Spec}}
\newcommand{\Tor}{\operatorname{Tor}}

\newcommand{\bsy}{\boldsymbol{y}}

\newcommand{\vf}{\varphi}

\newcommand{\lotimes}{\otimes^{\mathbf L}}

\newcommand{\xra}{\xrightarrow}
\newcommand{\lra}{\longrightarrow}

\newcommand{\sfD}{\mathsf D}

\newcommand{\bsx}{\boldsymbol{x}}

\newcommand{\fm}{\mathfrak{m}} 
\newcommand{\fp}{\mathfrak{p}}
\newcommand{\fn}{\mathfrak{n}}

\begin{document}

\title[Modules of finite flat dimension and the Frobenius endomorphism]{Detecting finite flat dimension of modules via \\ iterates of the Frobenius endomorphism}

\author[D.\,J.\, Dailey]{Douglas J. Dailey}

\address{D.J.D. University of Dallas, Irving, Texas 75062, U.S.A.}

\email{ddailey@udallas.edu}

\urladdr{http://www.udallas.edu}

\author[S.\,B.\ Iyengar]{Srikanth B. Iyengar}

\address{S.B.I. University of Utah, Salt Lake City, UT 84112, U.S.A.}

\email{iyengar@math.utah.edu}

\urladdr{http://www.math.utah.edu/~iyengar}

\author[T.\ Marley]{Thomas Marley}

\address{T.M. University of Nebraska-Lincoln, Lincoln, NE 68588, U.S.A.}
\email{tmarley1@unl.edu}

\urladdr{http://www.math.unl.edu/~tmarley}

\thanks{S.B.I.\ was partly supported by NSF grant DMS-1503044.}

\date{\today}

\bibliographystyle{amsplain}



\begin{abstract} 
It is proved that a module $M$ over a Noetherian ring $R$ of positive characteristic $p$ has finite flat dimension if there exists an integer $t\ge 0$ such that $\Tor_i^R(M, {}^{f^{e}}\!R)=0$ for $t\le i\le t+\dim R$ and infinitely many $e$. This extends results of  Herzog, who proved it when $M$ is finitely generated. It is also proved that when $R$ is a Cohen-Macaulay local ring, it suffices that the  Tor vanishing holds for one $e\ge \log_{p}e(R)$, where $e(R)$ is the multiplicity of $R$.
\end{abstract}

\maketitle

\section{Introduction}

The Frobenius endomorphism $f\colon R\to R$ of a commutative Noetherian local ring $R$ of prime characteristic $p$ is an effective tool for understanding the structure of such rings and the homological properties of finitely generated modules over them.  A paradigm of this is a result of Kunz \cite{Ku} that $R$ is regular  if and only $f^e$ is flat for some (equivalently, every) integer $e\ge 1$.  Our work is motivated by the following module-theoretic version of Kunz's result:

\medskip

\emph{There exists an integer $c$ such that for any finitely generated $R$-module $M$ the following statements are equivalent:
\begin{enumerate}[\quad\rm(1)] 
\item The flat dimension of $M$ is finite.
\item $\Tor_i^R(M, {}^{f^{e}}\!R)=0$ for all positive integers $i$ and $e$.
\item $\Tor_i^R(M, {}^{f^{e}}\!R)=0$ for all $i>0$ and infinitely many $e>0$.
\item $\Tor_i^R(M, {}^{f^{e}}\!R)=0$ for $\depth R+1$ consecutive values of $i>0$ and some $e>c$.
\end{enumerate}
}
\medskip

Peskine and Szpiro~\cite{PS} proved that (1)$\Rightarrow$(2), Herzog~\cite{He} proved that (3)$\Rightarrow$(1), and Koh and Lee \cite{KL} proved that (4)$\Rightarrow$(1).  Recently, the third author and M. Webb \cite[Theorem 4.2]{MW} proved the equivalence of conditions (1), (2), and (3) for all $R$-modules, even infinitely generated ones. In their work, the argument for (3)$\Rightarrow$(1) is quite technical and heavily dependent on results of Enochs and Xu~\cite{EX} concerning flat cotorsion modules and minimal flat resolutions.  

In this work we give another proof of \cite[Theorem 4.2]{MW} that circumvents \cite{EX}; more to the point, it yields a stronger result and sheds additional light also on the finitely generated case. See \cite{AF} for the definition of the flat dimension of a complex.

\begin{theorem}
\label{th:main} 
Let $R$ be a Noetherian local ring of prime characteristic $p$ and $M$ an $R$-complex with $s:=\sup \hh_{*}(M)$ finite. The following conditions are equivalent:
\begin{enumerate}[\quad\rm(1)]
\item The flat dimension of $M$ is finite.
\item $\Tor_i^R(M, {}^{f^{e}}\!R)=0$ for all $i>s$ and $e>0$.
\item There exists an integer $t\ge s$ such that $\Tor_i^R(M, {}^{f^{e}}\!R)=0$ for $t\le i\le t+\dim R$ and for infinitely many $e$.
\end{enumerate}
Moreover, when $R$ is Cohen-Macaulay of multiplicity $e(R)$, it suffices that the vanishing in \emph{(3)} holds for one $e\ge \log_pe(R)$.
\end{theorem}

This result is proved in Section~\ref{sec:proof}. The key element  in our proofs is the use of homotopical Loewy lengths of complexes, in much the same way as in the work of the second author and Avramov and C.~Miller~\cite[Section 4]{AHIY}. Using rather different methods, Avramov and the second author~\cite{AI} have proved that the last part of the theorem above holds without the Cohen-Macaulay hypothesis,  but with a different lower bound on $e$.

\section{Homotopical Loewy length} 
\label{sec:local}
In this section we collect results on homotopical Loewy length,  and their corollaries, needed in our proof of Theorem~\ref{th:main}. Throughout $(R,\fm,k)$ will be a local (this includes commutative and Noetherian) ring, with maximal ideal $\fm$ and residue field $k$; there is no restriction on its characteristic. We adopt the terminology and notation of \cite[Section 2]{AHIY} regarding complexes and related constructs.  In particular, given $R$-complexes $M$ and $N$, the notation $M \simeq N$ means that $M$ and $N$ are isomorphic in $\sfD(R)$, the derived category of $R$-modules.

The \emph{Loewy length} of an $R$-complex $M$ is the number
\[
\lol_R(M):= \inf\{n\in \mathbb N\mid \fm^n M=0\}.
\]
Following \cite[6.2]{AIM}, the \emph{homotopical Loewy length} of an $R$-complex $M$  is the number
\[
\lol_{\sfD(R)}(M):= \inf \{\lol_R(V) \mid M \simeq V \text{ in } \sfD(R)\}.
\]

Given a finite sequence $\bsx\subset R$ and an $R$-complex $M$, we write $\kos{\bsx}M$ for the Koszul complex on $\bsx$, with coefficients in $M$.   The result below extends, with an identical proof, \cite[Proposition 4.1]{AHIY} and \cite[Theorem 6.2.2]{AIM} that deal with the case when $\bsx$ generates $\fm$.
 
\begin{proposition}
\label{pr:loewy}
Let $\bsx$ be a finite sequence in $R$ such that $\length_R(R/\bsx R)$ is finite. For each $R$-complex $M$ there are inequalities
\[
\lol_{\sfD(R)} \kos{\bsx}M \leq \lol_{\sfD(R)}\kos{\bsx}R <\infty\,.
\]
\end{proposition}

\begin{proof} 
Let $I=(\bsx)$ and $K=\kos{\bsx}R$.  For each $i$, consider the subcomplex $C^i$ of $K$
\[
 \cdots \to I^{i-2}K_{2} \to I^{i-1}K_1\to I^iK_0\to 0.
\]  
Since $I^{i}$ annihilates $K/C^i$ and $I$ is $\fm$-primary, it follows that $\lol_{R}(K/C^{i})$ is finite for each $i$.  There exists an $r$ such that $C^i$ is acyclic for all $i\ge r$; cf. \cite[Proposition]{EF}. Thus, for $i\ge r$ the natural map $K\to K/C^i$ is a quasi-isomorphism, and hence the homotopical Loewy length of $K$ is finite. The inequality $\lol_{\sfD(R)}(K\otimes_{R}M) \le \lol_{\sfD(R)}K$ can be verified exactly as in the proof of \cite[Proposition 4.1]{AHIY}.
\end{proof}

The following invariant  plays an important role in what follows.
\[
c(R):=\inf\{\lol_{\sfD(R)} \kos{\bsx}R \mid\text{$\bsx$ is an s.o.p.\,for $R$}\}.
\]
Proposition~\ref{pr:loewy} yields that $c(R)$ is finite for any $R$.  For our purposes, we need a uniform bound on $c(R_{\fp})$, as $\fp$ varies over the primes ideals in $R$. We have been able to establish this only for Cohen-Macaulay rings; this is the content of the next result, where $e(R)$ denotes the multiplicity of $R$.

\begin{lemma} 
\label{le:CM}
Let $(R,\fm,k)$ be local ring with $k$ infinite. When $R$ is Cohen-Macaulay, there is an inequality $c(R_{\fp})\le e(R)$ for each $\fp$ in $\Spec R$.
\end{lemma}

\begin{proof}
By a result of Lech \cite{L}, one has $e(R)\ge e(R_{\fp})$ for all $\fp\in \Spec R$. Moreover, it is easy to verify that since $k$ is infinite, so is $k(\fp)$ for each $\fp$. It thus suffices to verify that $c(R)\le e(R)$.

Let $\bsx$ be a s.o.p. of $R$ that is a minimal reduction of $\fm$; this exists because $k$ is infinite; see \cite[Proposition 8.3.7]{SH}.  Then there are inequalities
\[
e(R)=\length_R(R/(\bsx)) \ge  \lol_R (R/(\bsx))=\lol_{\sfD(R)}\kos{\bsx}R \ge c(R).
\]
For the first equality see, for example, \cite[Proposition~11.2.2]{SH}, while holds the second equality holds because $\kos{\bsx}R\simeq R/(\bsx)$; both need the hypothesis that $R$ is Cohen-Macaulay. 
\end{proof}

The preceding result bring up the following question; its import for the results in this paper will become apparent in the  proof of~Theorem \ref{th:main}.

\begin{question} 
Is $\sup \{ c(R_{\fp})\mid \fp\in \Spec R\}$ finite for any Noetherian ring $R$?
\end{question}

The proof of Lemma~\ref{le:CM} is not likely to be of help in answering this question.

\begin{remark}
Let $\bsx$ be a finite sequence in a local ring $(R,\fm,k)$. Since the ideal $(\bsx)$ annihilates $\hh_{*}(\kos{\bsx}R)$, it is immediate from definitions that there is an inequality
\[
\lol_R (R/(\bsx))\leq \lol_{\sfD(R)}\kos{\bsx}R\,.
\]
Equality holds when $\bsx$ is a regular sequence, for then $R/(\bsx)\simeq \kos{\bsx}R$; this is the main reason for the Cohen-Macaulay hypothesis in Lemma~\ref{le:CM}. The inequality can be strict in general.

For example, if $(\bsx)=\fm$, then $\lol_{R}(R/\fm) = 1$, whilst $\lol_{\sfD(R)}\kos{\bsx}R=1$ exactly when $R$ is regular; see \cite[Corollary~6.2.3]{AIM}. 

Here is an example where the inequality is strict for $\bsx$ a s.o.p.

Let $R:=k[|x,y|]/(x^{n}y,y^{2})$, where $n\ge 1$ is an integer. The residue class of $x$ in $R$ is a s.o.p., and  the Loewy length of $R/(x)$ equals $2$. We claim that the homotopical Loewy length of $\kos xR$ is $n+1$.

Indeed, the following subcomplex of $\kos xR$ 
\[
A:=0\to (x^{n})\to (x^{n+1})\to 0
\]
is acyclic so one has $\kos xR\xra{ \simeq } \kos xR/A$. Since
\[
\kos xR/A = 0\lra  \frac R{(x^{n})}\lra \frac R{(x^{n+1})} \lra 0
\]
and the Loewy length of this complex is $n+1$, it follows that $\lol_{\sfD(R)}(\kos xR)\le n+1$. On the other hand
\[ 
\hh_{1}(\kos xR)=(x^{n-1}y)\subset R\,.
\]
Suppose $\kos xR\simeq V$  for some $R$-complex $V$. Since $\kos xR$ is a finite complex of free $R$-modules, there must exist a morphism $f\colon \kos xR\to V$ of $R$-complexes with $\hh_{*}(f)$ an isomorphism. The map $f_{1}\colon K_{1}=R\to V_{1}$ satisfies 
\[
0\ne f_{1}(x^{n-1}y)=x^{n-1}yf(1)
\]
It follows that $x^{n-1}y\cdot V_{1}\ne 0$, and hence that $\lol_{R}V\ge \lol_{R}(V_{1})\ge n+1$.
\end{remark}

As in \cite[Proposition 4.3(2)]{AHIY} one can apply Proposition~\ref{pr:loewy} to local homomorphisms to obtain an isomorphism relating Koszul homologies.

\begin{proposition}
\label{pr:tor-iso} 
\pushQED{\qed}
Let $(S,\fn,l)$ be a local ring, $\bsy$ a finite sequence of elements in $S$ such that the ideal $(\bsy)$ is $\fn$-primary, and set $c:=\lol_{\sfD(S)}\kos{\bsy}S$. 

If $\vf\colon (R,\fm, k)\to (S,\fn,l)$ is a local homomorphism satisfying $\fm S\subseteq \fn^c$, then for each $R$-complex $M$, there exists an isomorphism of graded $k$-vector spaces
\[
\Tor^R_*(M, \kos{\bsy}S) \cong \Tor^R_*(M,k)\otimes_k \hh_*(\kos{\bsy}S)\,.\qedhere
\]
\end{proposition}
%
%\begin{proof}  
%Let $K=\kos{\bsy}S$. By definition of homotopical Loewy length there exists an $S$-complex $V$ such that $K\simeq V$ in $\sfD(S)$ and $\fn^cV=0$.   Hence, $\fm V=0$ as well and $V\simeq \hh(V)$ in $\sfD(R/\fm)$, hence in $\sfD(R)$.  Thus, $K\simeq \hh(K)$ in $\sfD(R)$.  Consequently, one gets the following isomorphisms in $\sfD(R)$:
%\[
%M \lotimes_R K \simeq M \lotimes_R \hh(K) \simeq (L\lotimes_R k)\otimes_k \hh(K).
%\]
%Taking homology and applying the K\"unneth isomorphism gives the desired result.
%\end{proof}

We also need the following routine computation.

\begin{lemma}
\label{le:lem1} 
Let $\vf\colon R\to S$ be a homomorphism of rings, $\bsy=y_1,\dots, y_d$ a sequence of elements in $S$.  Let M be an $R$-complex and  $t$ an integer such that $\Tor_{i}^R(M, S)=0$ for $t\le i\le t+d$. Then $\Tor_{t+d}^R(M, \kos{\bsx}S)=0$. \qed
\end{lemma} 

%\begin{proof} 
%We argue by induction on $r$.  Let $K=\kos{\bsx}S$ and $C=\kos{\bsxp}S$, where $\bsxp=x_1,\dots, x_{r-1}$.  (For $r=1$ we have $C=S$.)   Then $K=C\otimes_S \kos{x_r}S$.  The induction hypothesis yields $\Tor^R_{t+r-1}(M, C)=\Tor^R_{t+r}(M, C)=0$.  Let $F$ be a semi-free resolution of $M$ and set $B=F\otimes_R C\simeq M\lotimes_R C$.  Then  we have isomorphisms in $\sfD(R)$
%\[
%F\lotimes_R K\simeq F\otimes_R C\otimes_S \kos{x_r}S\simeq B\otimes_S \kos{x_r}S.
%\]
%From the  exact sequence of complexes 
%\[
%0\to B\to B\otimes_S \kos{x_r}S\to \Sigma B\to 0
%\]
%we obtain the exact sequence
%\[
%\cdots \to \Tor^R_{i}(M, C)\to \Tor^R_{i}(M,K)\to \Tor_{i-1}^R(M, C)\to  \cdots .
%\]
%The result now follows.
%\end{proof}

Applied to (an appropriate composition of the Frobenius endomorphism) the next result  yields an analogue of \cite[Proposition 2.6]{KL} for complexes.  The number of consecutive vanishing of Tor  required in the case of  modules is not optimal ($\dim R+1$  as compared to $\depth R + 1$ in \cite{KL}), but the proof we give applies to complexes whose homology need not be  finitely generated.

\begin{lemma}
\label{le:prop1}  
Let $\vf\colon (R,\fm,k) \to (S,\fn,l)$ be a homomorphism of local rings such that $\vf(\fm) \subseteq \fn^{c(S)}$. 
Let $M$ be an $R$-complex.

If there is an integer $t$ such that $\Tor_i^R(M,S)=0$ for $t\le i\le t+\dim S$, then 
\[
\Tor_{t+\dim S}^R(M,k)=0\,.
\]
If moreover the $R$-module $\hh(M)$ is finitely generated and $t\geq \sup \hh(M)-\dim S$, the flat dimension of $M$ is at most $t+\dim S$.

\end{lemma}

\begin{proof} 
Set $d:=\dim S$ and let $\bsy$ be an s.o.p of $S$ such that $c(S)=\lol_{\sfD(S)}\kos {\bsy}S$. The hypothesis on $\vf$ and Lemma \ref{le:lem1} yield $\Tor_{t+d}^R(M, \kos {\bsy}S)=0$.  It then follows from Proposition \ref{pr:tor-iso} that $\Tor_{t+d}^R(M,k)=0$, since $\hh_0(\kos {\bsy}S)\neq 0$. 

Given this, and the additional hypotheses on $\hh(M)$ and $t$, the desired result follows from the existence of minimal resolutions; see \cite[Proposition~5.5(F)]{AF}.
\end{proof}

\section{Finite flat dimension} 
\label{sec:proof}
This section contains a proof of Theorem~\ref{th:main}. In preparation, we recall that an $R$-complex has \emph{finite flat dimension} if it is isomorphic in $\sfD(R)$ to a bounded complex of flat $R$-modules. The following result is \cite[Theorem 4.1]{CIM}.

\begin{remark}
\label{re:CIM}
Let $R$ be a Noetherian ring and $M$ an $R$-complex.  If there exists an integer $n\ge \sup \hh_{*}(M)+\dim R$ with $\Tor_n^{R_{\fp}}(M_{\fp},k(\fp))=0$ for all $\fp\in \Spec R$ then the flat dimension of $M$ is finite.
\end{remark}

In what follows, given an endomorphism $\phi\colon R\to R$ and an $R$-complex $M$, we write ${}^{\phi}\!M$ for $M$ viewed as an $R$-complex via $\phi$. 

\begin{proof}[Proof of Theorem \ref{th:main}]
Recall that $R$ is a Noetherian ring of prime characteristic $p$ and $f\colon R\to R$ is the Frobenius endomorphism.  

(1)$\Rightarrow$(2): Fix an integer $e\ge 1$ and set $r:=\sup \Tor_{*}^R(M, {}^{f^{e}}\!R)$. Since $\flatdim_R M$ is finite,  $r<\infty$ holds. The desired result is that $r\le s$.

Pick a prime ideal $\fp$ associated to $\Tor_{r}^R(M, {}^{f^{e}}\!R)$. Since Frobenius commutes with localization one has
\[
\Tor_{r}^R(M, {}^{f^{e}}\!R)_{\fp} \cong \Tor_{r}^{R_{\fp}}(M_{\fp}, {}^{f^{e}}\!R_{\fp})
\]
as $R_{\fp}$-modules. Moreover $\flatdim_{R_{\fp}}M_{\fp}$ is finite. Thus replacing $R$ and $M$ by their localizations at $\fp$ we get that the maximal ideal of $R$ is associated to $\Tor_{r}^R(M, {}^{f^{e}}\!R)$; that is to say, the depth of the latter module is zero.

The next step uses some results concerning depth for complexes; see~\cite{FI}. Given the conclusion of the last paragraph, \cite[2.7]{FI} yields the last equality below.
\begin{align*}
\depth R - \sup\Tor_{*}^{R}(k,M)
	&= \depth_{R}({}^{f^{e}}\!R) - \sup\Tor_{*}^{R}(k,M) \\
	&= \depth_{R}(M\lotimes_{R}{}^{f^{e}}\!R)\\
	&=-r
\end{align*}
The second one is by \cite[Theorem 2.4]{FI}. The same results also yield
\[
\depth R - \sup\Tor_{*}^{R}(k,M) = \depth_{R}M \ge - \sup\hh_{*}(M) =-s
\]
It follows that $-r\ge -s$, that is to say, $r\leq s$, which is the desired conclusion.

\medskip

(3)$\Rightarrow$(1):  Let $d=\dim R$. By Remark~\ref{re:CIM},  it suffices to verify that
\begin{equation}
\label{eq:proof}
\Tor^{R_{\fp}}_{t+d}(M_{\fp}, k(\fp))=0\quad\text{for all $\fp\in \Spec R$}.
\end{equation}
Fix $\fp\in \Spec R$ and choose $e$ such that $p^e\ge c(R_{\fp})$ and $\Tor^{R}_{i}(M,{}^{f^{e}}\!R)=0$ for $t\le i\le t+d$; such an $e$ exists by our hypothesis.  As the Frobenius map commutes with localization, one gets
\[
\Tor_i^{R_{\fp}}(M_{\fp}, {}^{f^{e}}\!(R_{\fp}))=0\quad \text{for $t\le i \le t+d$}.
\]
The choice of $e$ ensures that $f^{e}(\fp R_{\fp})\subseteq \fp^{c(R_{\fp})}R_{\fp}$. Thus, Lemma~\ref{le:prop1} applied to the Frobenius endomorphism $R_{\fp}\to R_{\fp}$ yields $\Tor_{t+d}^{R_{\fp}}(M_{\fp}, k(\fp))=0$, as desired. 

\medskip

Assume now that $R$ is Cohen-Macaulay and that the vanishing in (3) holds for some $e\ge \log_{p}e(R)$. It is a routine exercise to verify that the hypotheses remain unchanged, and that the desired conclusion can be verified, after passage to faithfully flat extensions. One can thus assume that the residue field $k$ is infinite; see \cite[IX.37]{Bo}. Then, by the choice of $e$ and Lemma~\ref{le:CM}, one gets that $p^{e}\ge c(R_{\fp})$ for each $\fp$ in $\Spec R$. Then one can argue as above to deduce that \eqref{eq:proof} holds,  and that yields the finiteness of the flat dimension of $M$.
\end{proof}

%\begin{remark} We note that proof of (3)$\Rightarrow$(1) above does not rely upon the theory of minimal flat resolutions as developed by Enochs and Xu~\cite{EX}.  However,
%the argument that  (1)$\Rightarrow$(2) invokes \cite[Corollary 3.5(a)]{MW}, which does rely on results in \cite{EX}.
%\end{remark}

\end{document}